\input ./style/arxiv-vmsta.cfg
\documentclass[numbers,compress,v1.0.1]{vmsta}
\usepackage{vtexbibtags}

\volume{6}
\issue{2}
\pubyear{2019}
\firstpage{209}
\lastpage{225}
\aid{VMSTA129}
\doi{10.15559/19-VMSTA129}
\articletype{research-article}

\startlocaldefs

\allowdisplaybreaks

\newtheorem{thm}{Theorem}

\theoremstyle{definition}

\newtheorem{exmp}{Example}
\newtheorem{assumpA}{Assumptions}

\ifdefined\HCode 
\def\index#1{}
\else 
\fi
\endlocaldefs

\begin{document}

\begin{frontmatter}
\pretitle{Research Article}

\title{Testing proportional hazards for specified covariates}
\author{\inits{V.}\fnms{Vilijandas}~\snm{Bagdonavi\v{c}ius}\thanksref
{cor1}\ead[label=e1]{vilijandas.bagdonavicius@mif.vu.lt}}
\author{\inits{R.}\fnms{R\={u}ta}~\snm{Levulien\.{e}}\ead
[label=e2]{ruta.levuliene@mif.vu.lt}}
\thankstext[type=corresp,id=cor1]{Corresponding author.}
\address{Faculty of Mathematics and Informatics, \institution{Vilnius
University}, Naugarduko 24,
Vilnius~LT-03225, \cny{Lithuania}}

\markboth{V. Bagdonavi\v{c}ius, R. Levulien\.{e}}{Testing proportional
hazards for specified covariates}

\begin{abstract}
Tests for proportional hazards assumption concerning
specified covariates or groups of covariates are proposed. The class of
alternatives is wide: log-hazard rates under different values of
covariates may cross, approach, go away. The data may be right
censored. The limit distribution of the test statistic is derived.
Power of the test against approaching alternatives is given. Real
data examples are considered.
\end{abstract}

\received{\sday{2} \smonth{10} \syear{2018}}
\revised{\sday{9} \smonth{1} \syear{2019}}
\accepted{\sday{9} \smonth{1} \syear{2019}}
\publishedonline{\sday{15} \smonth{2} \syear{2019}}
\end{frontmatter}


%

\section{Introduction}

The most known semi-parametric model for analysis of failure time
regression data\index{regression data}
is the proportional hazards\index{proportional hazards (PH)} (PH) model.
There are many tests for the PH model from right censored data given by
Cox~\cite{Cox1972}, Moreau et\,al.~\cite{Moreau1985}, Lin~\cite{Lin1991}, Nagelkerke et\,al. \cite{Nagelkerke1984},
Grambsch and Therneau~\cite{Grambsch1994}, Quantin et\,al.~\cite{Quantin1996}, Bagdonavi\v{c}ius {et\,al.}~\cite{Bagdonavicius2004}.
Tests for specified covariates are given in Kvaloy and Neef~\cite{Kvaloy2004}, Kraus~\cite{Kraus2008}.

We consider tests for proportional hazards assumption concerning
specified covariates or groups of covariates.\index{covariate} The class of alternatives
is wide: log-hazard rates under different values of covariates\index{covariate} may
cross, approach, go away. The data may be right censored. The limit
distribution of the test statistics
is derived. Power of the test under approaching alternatives\index{approaching alternatives} is given.
Real examples are considered.

\section{Modeling non-proportional hazards for specified covariates}

Let $S(t|z)$,
$\lambda(t|z)$ and $ \Lambda(t|z)$ be the survival, hazard
rate and cumulative hazard functions\index{cumulative hazard functions} under a
$m$-dimensional explanatory variable (or covariate\index{covariate}) $z=(z_1,\ldots,z_m)^T$.

Let us consider the proportional hazards\index{proportional hazards (PH)} (PH) model:
\[
\lambda(t|z)=e^{\beta^Tz}\, \lambda(t),
\]
where
$\lambda(t)$ is unknown baseline hazard function, the parameter
$\beta=(\beta_1,\ldots,\beta_m)^T$ is $m$-dimensional.

Under the PH model the ratios of hazard functions\index{hazard functions} under any two
different explanatory variables are constant over time, i.e. they are
proportional.

Suppose that proportionality of hazard functions\index{hazard functions} with respect to a
specified covariate $z_j$ or a group of covariates
$z_{j_1},\ldots,z_{j_k}$ may be violated.
Our purpose is to detect such violations. So we seek a wider model
which includes not only possibility of constant hazard rates\index{hazard rates} ratios but also
time-varying ratios.

We propose a model of the form
\begin{equation}
\label{model}
\lambda(t|z)=g\bigl(z,\Lambda(t),\beta,\gamma_j\bigr)
\, \lambda(t), \quad\Lambda(t)=\int_0^t
\lambda(u)du,
\end{equation}
where
\[
g(z,\Lambda(t),\beta,\gamma_j)=\frac{e^{\beta^Tz+\Lambda(t)e^{\gamma_j
z_j}}}{1+e^{\gamma_j z_j}[e^{\Lambda(t)e^{\gamma_j z_j}}-1]}.
\]
$\beta=(\beta_1,\ldots,\beta_m)^T$ and $\gamma_j$ are unknown regression
parameters, $\lambda(t)$ and $\Lambda(t)$ are
unknown baseline hazard and cumulative hazard, respectively\xch{.}{..}

The PH model is a particular case of this model when $\gamma_j=0$. The
model (\ref{model}) is very wide 
as compared to the PH model.

Indeed, suppose that two different values $z^{(1)}$ and $z^{(2)}$ of
the covariate vector\index{covariate ! vector} differ only by the value of the covariate\index{covariate} $z_j$
and denote by $c(t)$
the ratio of hazard rates.\index{hazard rates} For the model (\ref{model})
\[
c(0)=e^{\beta_j(z_j^{(2)}-z_j^{(1)})},\qquad c(\infty)=e^{(\beta_j-\gamma
_j)(z_j^{(2)}-z_j^{(1)})}.
\]
So this model gives a large choice of alternatives: the ratio of hazard
rates $c(t)$\index{hazard rates} can vary from any $a>0$ to any $b>0$, whereas $c(t)$ is
constant under the PH model. The difference of logarithms of hazard
rates\index{hazard rates} under different values of the covariate\index{covariate} $z_j$ is constant under
the PH model, so the logarithms of hazard rates\index{hazard rates} as time
functions are
parallel, whereas in dependence on the values of its parameters the
model (\ref{model}) includes also the possibilities for these functions to
approach, to go away, to intersect. So the model (\ref{model}) may help to detect
non-proportionality of hazard rates\index{hazard rates} (or, equivalently, non-parallelism
of log-hazard functions) in above mentioned directions. Other
directions are very rare in real data.

We do not discuss application of the model for analysis of survival
regression data (which is a subject of another article). Such analysis
could be done if the PH model would be rejected. Here the model is used
only as a generator of a wide class of alternatives to the PH model.

\section{Test statistic}

Let us consider right censored failure time regression data:\index{regression data}
\[
(X_1,\delta_1,z^{(1)}),\ldots, (X_n,\delta_n,z^{(n)}),
\]
where
\[
X_i=T_i \wedge C_i,\quad\delta_i=\mathbf{1}_{{\{T_i\leq
C_i\}}},
\]
$T_i$ are failure times and $C_i$ are censoring times.\index{censoring time} $X_i$ is
observation time of the $i$th unit, the event $\delta_i=1$ indicates
that $X_i$ is the failure time $T_i$, and the event $\delta_i=0$
indicates that $X_i$ is the censoring time $C_i$.\index{censoring time}

Set
\begin{alignat*}{4}
N_i(t) &{}= \mathbf{1}_{\{X_i\leq t,\delta_i=1\}},& Y_i(t)&{}=\mathbf{1}_{\{X_i\geq t\}},\\
N(t) &{} =\sum_{i=1}^nN_i(t),& \quad Y(t)&{}=\sum_{i=1}^nY_i(t);
\end{alignat*}
here
$\mathbf{1}_A$ denotes the indicator of the event $A$. The processes
$N_i(t)$ and $N(t)$ are right-continuous counting process showing the
numbers of observed failures in the time interval $[0,t]$ for the
$i$th unit and for all $n$ units, respectively. $Y_i(t)$ and $Y(t)$ are
decreasing (not strongly) left-continuous stochastic processes showing
the numbers of units which are still not-failed and not-censored just
prior to $t$ for the $i$th and for all units, respectively.

Suppose
that the survival distribution of the $i$ object given $z^{(i)}$ is
absolutely continuous with the survival functions $S_i(t)$ and the
hazard rates\index{hazard rates} $\lambda_i(t)$.

Suppose that the multiplicative intensities\index{multiplicative intensities} model is
verified: the compensators of the counting processes $N_i$ with
respect to the history of the observed processes are $\int
Y_i\lambda_i du$. It is equivalent to the assumption that
for any $i$ and $t$: $P(X_i > t)>0$ with almost all
$s\in[0,t]$
\begin{align*}
&{} \lim_{h \downarrow 0}\frac{1}{h}P\bigl\{T_i \in[s,s+h)|X_i \geq s,x^{(i)}\bigr\}\\
&\quad {}= \lim_{h
\downarrow0}\frac{1}{h}P\bigl\{T_i \in[s,s+h)|T_i \geq s,x^{(i)}\bigr\}=\lambda_i(s).
\end{align*}
It means that for almost all $s\in[0,t]$ the risk to fail just after
the time $s$ given that units were non-censored and did not fail to
time $s$ is equal
to the risk to fail just after the time $s$ given that units did not
fail to time $s$ when censoring\index{censoring} does not exist.
So censoring\index{censoring} has no influence to the risk of failure.

Information about time-to-failure distribution contains the points,
where the counting processes $N_i$ have jumps.
A jump at the point $t$ is possible if $Y_i(t)=1$. Partial information
give the points where the counting processes $N_i$ have not jumps
but the predictable processes have jumps.

The multiplicative intensities\index{multiplicative intensities}
model is verified in the case of type I, type II, independent
random censoring.\index{independent random censoring}

In the parametric case with known $\lambda$ the unknown
finite-dimensional parameter consists of
the parameters $\beta$ and
$\gamma_j$. The component of the parametric score function $U$
corresponding to $\gamma_j$ has the form
\begin{equation}\label{eq2}
U_{j}(\beta,\gamma_j)=\sum_{i=1}^{n}
\int^\infty_0\bigl\{w_j^{(i)}(u,
\Lambda,\beta,\gamma_j)-\tilde{E}_j(u, \Lambda,\beta,
\gamma_j)\bigr\}\,dN_i(u),
\end{equation}
where
\begin{gather*}
w_j^{(i)}(v,\Lambda,\beta,\gamma_j)=\frac{\partial}{\partial\gamma_j}
\log\bigl\{g(z^{(i)}, \Lambda(v),\beta,\gamma_j)\bigr\},
\\
\tilde{E}(v,\Lambda,\beta,\gamma_j)=\frac{\tilde{S}^{(1)}(v,\Lambda
,\beta,\gamma_j)}{\tilde{S}^{(0)}(v,\Lambda,\beta,\gamma_j)},
\\
\tilde{S}^{(0)}(v,\Lambda,\beta,\gamma_j)= \sum^n_{i=1}Y_i(v)g\bigl\{
z^{(i)}, \Lambda(v),\beta,\gamma_j\bigr\},
\\
\tilde{S}^{(1)}(v,\Lambda,\beta,\gamma_j)=\sum^n_{i=1}
Y_i(v)\,\frac{\partial}{\partial\gamma_j} g\bigl\{z^{(i)}, \Lambda(v),\beta
,\gamma_j\bigr\}.
\end{gather*}
We consider semiparametric case when the baseline hazard rate $\lambda$
is unknown.
The test statistic is constructed in the following way. In the
expression of $U_j$ the parameter
$\beta$ is replaced by its partial likelihood estimator
$\hat\beta$ under the Cox model,\index{Cox model} the parameter
$\gamma_j$ is replaced by $0$, and the baseline cumulative intensity
$\Lambda$ is replaced by the
Breslow estimator (see Andersen \xch{{et\,al.}}{{\it et al}}~\cite{Andersen1993})
\[
\hat\Lambda(t)=\int_0^t\frac{d
N(u)}{S^{(0)}(u,\hat\beta)}.
\]
The estimator $\hat\beta$ verifies the equation
\begin{equation}
\label{esteq} \sum_{i=1}^n\int
_0^{\infty} \bigl\{z^{(i)}-E(t,\hat\beta)
\bigr\} dN_i(t)=0,
\end{equation}
where
\begin{gather*}
E(t,\beta)=\frac{S^{(1)}(t,\beta)}{S^{(0)}(t,\beta)},\\
{S}^{(0)}(t,\beta)=\sum^n_{i=1}Y_i(t)e^{\beta^Tz^{(i)}},\qquad
{S}^{(1)}(t,\beta)=\sum^n_{i=1}z^{(i)}Y_i(t)e^{\beta^Tz^{(i)}}.
\end{gather*}
Set $\hat F(t)=1-e^{-\hat\Lambda(t)}$. It is an estimator of the
baseline cumulative distribution function. After the above mentioned
replacement the following statistic is obtained:
\begin{equation}
\label{u_gamma} \hat U_{\gamma_j}=-\sum_{i=1}^n
\int_0^{\infty}\hat F(t) \bigl\{z_j^{(i)}-E_{j}(t,
\hat\beta)\bigr\} dN_i(t),
\end{equation}
where $E_j$ is the component of $E$ corresponding to the covariate\index{covariate}
$z_j$.\goodbreak

Let us consider asymptotic distribution of the statistic (\ref{u_gamma}) under the
PH model. Set
\begin{gather*}
{S}^{(2)}(u,\beta)=\sum^n_{i=1}(z^{(i)})^{\otimes2}Y_i(u)e^{\beta^Tz^{(i)}},\\
V(u,\beta)={{S}^{(2)}(u,\beta)}/{{S}^{(0)}(u,\beta)}-{E}^{\otimes2}(u,\beta),
\end{gather*}
where $A^{\otimes2}=AA^T$ for any matrix $A$. Denote by $V_j(u,\beta)$ the
$j$th column and by $V_{jj}(u,\beta)$ the $j$th
diagonal element of the matrix $V(u,\beta)$. Set
\begin{gather*}
\hat\Sigma_{jj}(t)=n^{-1}\int^{t}_0\hat F(u)V_{jj}(u,\hat\beta)d
N(u),\qquad\hat\Sigma_{j}(t)=n^{-1}\int^{t}_0\hat F(u)V_{j}(u,\hat\beta
)d N(u),
\\
\hat\Sigma(t)=n^{-1}\int^{t}_0V(u,\hat\beta)d N(u).
\end{gather*}
\begin{assumpA}\label{assumpA}
\begin{enumerate}
\item[a)] $(X_i,\delta_i,z^{(i)})$ are \xch{i.i.d.,}{i.i.d.}
\item[b)] $\sup\{t:Y_i(t)>0\}\stackrel{P}\rightarrow\tau>0$,
\item[c)] the covariates $z^{(i)}$ are bounded: $\xch{\|}{||}z^{(i)}\xch{\|}{||}\leq C$, $C>0$,
\item[d)] $\Lambda(\tau)<\infty$,
\item[e)] the matrix
\[
\Sigma(\beta)=\int_0^\tau v(u,\beta)s^{(0)}(u,\beta)d\Lambda(u)
\]
is positively definite; here
\begin{gather*}
s^{(i)}(t,\beta)=E S^{(i)}(t,\beta)/n, \qquad e(u,\beta)=s^{(1)}(u,\beta
)/s^{(0)}(u,\beta)\quad i=0,1,2,
\\
v(u,\beta)=s^{(2)}(u,\beta)/s^{(0)}(u,\beta)-e(u,\beta)(e(u,\beta))^T.
\end{gather*}
\end{enumerate}
\end{assumpA}

Assumption a) can be weakened considerably but we avoid writing
complicated formulas for easier reading. Assumption b) simply means
that at some finite time moment (perhaps very remote) observations are
stopped. This is a usual assumption for
asymptotic results to hold in survival analysis. Assumption c) also can
be weakened.
Assumption d) means that if censoring\index{censoring} would be absent then units might
survive after the moment $\tau$ with positive probability.
Assumption e) is needed to have non-degenerated asymptotic distribution
of the test statistic. This assumption is the usual assumption needed for
asymptotic normality of the maximum partial likelihood estimator $\hat
\beta$ of the regression parameter $\beta$.

Under Assumptions \ref{assumpA}~\cite{AndGill1982} there exists a neighborhood
$\Theta$ of $\beta$ such that
\begin{gather}
\sup_{b\in\Theta,u\in[0,\tau]}\biggl\Vert\frac {1}{n}S^{(i)}(u,b)-s^{(i)}(u,b)\biggr\Vert\stackrel{P} \rightarrow0 ,\nonumber\\
\sup_{b\in\Theta,u\in[0,\tau]}\bigl\Vert E(u,b)-e(u,b)\bigr\Vert\stackrel{P} \rightarrow0,\nonumber\\
\sup_{b\in\Theta,u\in[0,\tau]}\bigl\Vert V(u,b)-v(u,b)\bigl\Vert \stackrel{P} \rightarrow0 ,\label{eq5}
\end{gather}
{as} $n\to\infty$.

Assumption~\ref{assumpA}\,a) may be weakened assuming
that there exist
non-random $s^{(i)}$ such that the convergences\index{convergences} (\ref{eq5}) hold.

Convergences\index{convergences} (\ref{eq5}) and Assumption~\ref{assumpA}\,d) imply that
\begin{gather*}
\sup_{b\in\Theta,t\in[0,\tau]}\biggl\Vert\int_0^t(\frac{1}{n}S^{(i)}(u,b)-s^{(i)}(u,b))d\Lambda(u)\biggr\Vert=o_P(1),
\\
\sup_{b\in\Theta,t\in[0,\tau]}\biggl\Vert\int_0^t( E(u,b)-e(u,b))d\Lambda(u)\biggr\Vert=o_P(1),
\\
\sup_{b\in\Theta,t\in[0,\tau]}\biggl\Vert\int_0^t( V(u,b)-v(u,b))d\Lambda(u)\biggr\Vert=o_P(1).
\end{gather*}

\begin{thm}\label{thm1} Under Assumptions~\ref{assumpA} the following weak convergence holds:
\begin{equation}
\label{statistics} T=n^{-1/2}\hat U_{\gamma_j}/\sqrt{\hat
D_j}\stackrel{ {d}}\rightarrow N(0,1),\quad\mbox{as $n\to\infty$},
\end{equation}
where
\begin{equation}\label{eq7}
\hat D_j=\hat\Sigma_{jj}(\tau)-\hat\Sigma^T_{j}(
\tau)\hat\Sigma ^{-1}(\tau)\hat\Sigma_{j}(\tau).
\end{equation}
\end{thm}

\begin{proof}
Under the Cox model,\index{Cox model}
\[
N_i(t)=\int_0^tY_i(u)e^{\beta^Tz^{(i)}}d\Lambda(u)+M_i(t),
\]
where $M_i$ are
martingales with respect to the history generated by the data. So using
it and taking into account that
\[
\sum_{i=1}^nz_j^{(i)}Y_i(u)e^{\hat\beta^Tz^{(i)}}-E_{j}(u,\hat\beta)\sum_{i=1}^n Y_i(u)e^{\hat\beta^Tz^{(i)}}=0,
\]
the statistic $\hat U_{\gamma_j}$
can be written in the form
\begin{align*}
\hat U_{\gamma_j}&{}=  \sum_{i=1}^n
\int_0^{t}\hat F(u) \bigl\{ z_j^{(i)}-E_{j}(u,
\hat\beta)\bigr\} Y_i(u) \bigl(e^{\hat\beta^Tz^{(i)}}-e^{\beta
^Tz^{(i)}}
\bigr)d\Lambda(u)\\
 &\quad{} -\sum_{i=1}^n\int
_0^{t} \hat F(u) \bigl\{z_j^{(i)}-E_{j}(u,
\hat\beta )\bigr\} dM_i(u).
\end{align*}
Assumptions~\ref{assumpA}\,a)--d) and consistence of the partial likelihood estimator
$\hat\beta$ imply that
\begin{align*}
n^{-1/2}\hat U_{\gamma_j}(t)&{}=  n^{-1} \sum
_{i=1}^n\int_0^{t}
\hat F(u)\bigl\{z_j^{(i)}
\\
&\quad{}  -E_{j}(u,\hat\beta)\bigr\}\bigl(z^{(i)} \bigr)^T Y_i(u)e^{\hat\beta ^Tz^{(i)}}d\Lambda(u) \,n^{1/2}(\hat\beta-\beta)
\\
&\quad{} - n^{-1/2}\sum_{i=1}^n\int_0^{t} \hat F(u) \bigl\{ z_j^{(i)}-E_{j}(u,\hat\beta)\bigr\}dM_i(u)+o_P(1)
\end{align*}
uniformly on $[0,\tau]$.
Using convergences\index{convergences} (\ref{eq5}) and consistence of the estimator $\hat\beta$
we write the first integral in the form
\begin{align*}
& n^{-1} \sum_{i=1}^n\int_0^{t}\hat F(u)\bigl\{z_j^{(i)} -E_{j}(u,\hat\beta)\bigr\}
    (z^{(i)})^T Y_i(u)e^{\hat\beta^Tz^{(i)}}d\hat\Lambda(u)+o_P(1)\\
& \quad{}= n^{-1} \sum_{i=1}^n\int_0^{t}\hat F(u)\bigl\{z_j^{(i)} -E_{j}(t,\hat\beta)\bigr\}
    (z^{(i)})^T Y_i(u)e^{\hat\beta^Tz^{(i)}}\frac{dN(u)}{S^{(0)}(u,\hat\beta )}+o_P(1) \\
& \quad{}= \hat\Sigma^T_{j}(\tau)+o_P(1).
\end{align*}
Under Assumptions~\ref{assumpA} (now Assumption~\ref{assumpA}\,e) is crucial) the
following
expression holds:
\[
n^{1/2}(\hat\beta-\beta)=\hat\Sigma^{-1}(\tau)n^{-1/2}\sum^n_{i=1}\int^{\tau}_0\{z^{(i)}-{E}(u,\hat\beta)\}
dM_i(u)+o_p(1)
\]
uniformly on $[0,\tau]$ (see Andersen \xch{{et\,al.}}{{et al}}~\cite{Andersen1993},
Theorem VII.2.2, where even weaker assumptions instead of a)--d) are
used). It implies
\begin{align*}
n^{-1/2}\hat U_{\gamma_j}(t)&{}= \hat\Sigma^T_{j}(t)
\hat\Sigma ^{-1}(\tau)\, n^{-1/2}\sum^n_{i=1}\int^{\tau}_0
\bigl\{z^{(i)}-{E}(u,\hat\beta )\bigr\}dM_i(u)\\
&\quad{} -n^{-1/2}\sum^n_{i=1}\int^{t}_0\hat F(u) \bigl\{z_j^{(i)}-E_{j}(t,\hat \beta)\bigr\}dM_i(u)+o_p(1)\\
&{}=: M^*(t)+o_p(1)
\end{align*}
uniformly on $[0,\tau]$.

The predictable variation\index{predictable variation} of the local martingale $M^*(t)$ is
\begin{align*}
& \bigl\langle M^* \bigr\rangle(t)\\
&\quad{} =\hat\Sigma^T_{j}(t)\hat\Sigma^{-1}(\tau)\, n^{-1}\sum^n_{i=1}\int^{t}_0\bigl\{z^{(i)} -{E}(u,\hat\beta)
\bigr\}^{\otimes2}Y_i(u)e^{\beta^Tz^{(i)}} d\Lambda(u)\hat
\Sigma^{-1}(\tau)\hat\Sigma_{j}(t)\\
&\qquad{} - \hat\Sigma^T_{j}(t)\hat\Sigma^{-1}(\tau)
\, n^{-1}\sum^n_{i=1}\int
^{t}_0\hat F(u) \bigl\{z_j^{(i)}
-E_{j}(u,\hat\beta)\bigr\}  \bigl\{z^{(i)}
 -{E}(u,\hat\beta)\bigr\}\\
&\qquad
\times{}Y_i(u)e^{\beta^Tz^{(i)}}d\Lambda(u)
 \\
&\qquad{}
-n^{-1}\sum^n_{i=1}\int
^{t}_0\hat F(u) \bigl\{z_j^{(i)}-E_{j}(u,
\hat\beta)\bigr\} \bigl\{z^{(i)}
-{E}(u,\hat\beta)\bigr\}^T
\\
&\qquad
\times{}Y_i(u)e^{\beta^Tz^{(i)}} d
\Lambda(u) \hat \Sigma^{-1}(\tau)\hat\Sigma_{j}(t)
\\
&\qquad{} +n^{-1}\sum^n_{i=1}\int
^{t}_0\hat F^2(u) \bigl
\{z_j^{(i)}-E_{j}(u,\hat \beta)\bigr
\}^2 Y_i(u)e^{\beta^Tz^{(i)}}d\Lambda(u).
\end{align*}
Note that
\begin{align*}
&{}n^{-1}\sum^n_{i=1}\int^{t}_0\bigl\{z^{(i)}-{E}(u,\hat\beta)\bigr\}^{\otimes
2}Y_i(u)e^{\beta^Tz^{(i)}}d\Lambda(u)=\hat\Sigma(t)+o_P(1),\\
&{} n^{-1}\sum^n_{i=1}\int^{t}_0\hat F(u) \bigl\{z_j^{(i)}-E_{j}(u,
\hat\beta)\bigr\} \bigl\{z^{(i)}-{E}(u,\hat\beta)\bigr\} 
Y_i(u)e^{\beta^Tz^{(i)}}d\Lambda(u)\\
&\quad{} =\hat\Sigma_{j}(t)+o_P(1),\\
&{}n^{-1}\sum^n_{i=1}\int^{t}_0\hat F^2(u)
\bigl\{z_j^{(i)}-E_{j}(u,\hat\beta)\bigr\}^2 Y_i(u)e^{\beta^Tz^{(i)}}d\Lambda
(u)=\hat\Sigma_{jj}(t)+o_P(1).
\end{align*}
So the predictable variation\index{predictable variation}
%
\[
\langle M^* \rangle(t)=\hat\Sigma_{jj}(t)-\hat\Sigma^T_{j}(t)\hat\Sigma
^{-1}(\tau)\hat\Sigma_{j}(t)+o_P(1)\stackrel{P}\rightarrow\Sigma
_{jj}(t)-\Sigma^T_{j}(t)\Sigma^{-1}(\tau)\Sigma_{j}(t)
\]
uniformly on $[0,\tau]$; here $\Sigma_{jj}(t)$, $\Sigma_{j}(t)$, $\Sigma
(\tau)$ are the limits in probability of the random matrices $\hat\Sigma
_{jj}(t)$, $\hat\Sigma_{j}(t)$, $\hat\Sigma(\tau)$.

Under Assumptions~\ref{assumpA} for any $\varepsilon>0$ the predictable variation\index{predictable variation}
(see Andersen \xch{et\,al.}{et al}~\cite{Andersen1993}), Theorem VII.2.2, where even
weaker assumptions are used)
\[
\Biggl\langle n^{-1/2}\sum^n_{i=1}\int^{t}_0
\bigl\{z_j^{(i)}-E_{j}(t,\beta)\bigr\}\mathbf{1}_{\{\mid{z}^{(i)}_{j}
-{E}_{j}(u,\beta)\mid> \sqrt{n}\varepsilon\}}dM_i(u) \Biggr\rangle\stackrel
{P}\rightarrow0
\]
and $\beta$ can be replaced by $\hat\beta$ in the expression of the left
side because $\hat\beta$ is consistent.
Similarly
\[
\Biggl\langle n^{-1/2}\sum^n_{i=1}\int^{t}_0\hat F(u)
\bigl\{z_j^{(i)}-E_{j}(t,\hat\beta)\bigr\}\mathbf{1}_{\{\hat F(u)\mid{z}^{(i)}_{j}
-{E}_{j}(u,\hat\beta)\mid> \sqrt{n}\varepsilon\}}dM_i(u) \Biggr\rangle
\stackrel{P}\rightarrow0.
\]
Hence, the Lindeberg
conditions of the central
limit theorem for martingales (see Andersen \xch{et\,al.}{et al}~\cite{Andersen1993}) are
verified for $M^*$. We saw that the predictable variations\index{predictable variation} $ \langle
M^* \rangle(t)$ converge in probability to non-random non-degenerated matrices.
So the stochastic process $n^{-1/2}\hat U_{\gamma_j}(\cdot)$
converges in distribution to a zero mean Gaussian process on $[0,\tau
]$, in
particular
\[
T=n^{-1/2}\hat U_{\gamma_j}/\sqrt{\hat
D}\stackrel{\cal{D}}\rightarrow N(0,1)\quad\mbox{ as $n\to\infty$},
\]
here
\[
\hat D=\hat\Sigma_{jj}(\tau)-\hat\Sigma^T_{j}(\tau)\hat\Sigma^{-1}(\tau
)\hat\Sigma_{j}(\tau).\qedhere
\]
\end{proof}

{ The test:} the null hypothesis $H_0: \gamma_j=0$ is rejected with
approximate significance
level $\alpha$ if $|T|>z_{\alpha/2}$; here $z_{\alpha/2}$ is the $\alpha/2$
critical value of the standard normal distribution.

\section{Power under approaching alternatives\index{approaching alternatives}}

Let us consider the alternative model (\ref{model}) and suppose that alternatives
are approaching: $\gamma_j=\frac{c}{\sqrt{n}}$, where $c\neq0$ is a
fixed constant.

So the model
\begin{equation}
\label{appr} \lambda(t|z)=\frac{\exp\{\beta^Tz+\Lambda(t)\exp\{\frac{c}{\sqrt{n}}
z_j\}\}}{1+\exp\{\frac{c}{\sqrt{n}} z_j\}[\exp\{\Lambda(t)\exp\{\frac
{c}{\sqrt{n}} z_j\}\}-1]}\, \lambda(t), \quad\Lambda(t)=\int
_0^t\lambda(u)du,
\end{equation}
is considered. Denote by $\beta_0$ the true value of $\beta$ under the
model (\ref{appr}). The counting process $N_i(t)$ has the form
\[
N_i(t)=\int_0^t g\bigl\{z^{(i)}, \Lambda(u),\beta_0,c/\sqrt{n}\bigr\}d\Lambda
(u)+M_i(t),
\]
where $M_i$ are local
martingales with respect to the history generated by the data. Note
that
\[
g\bigl\{z^{(i)}, \Lambda(u),\beta_0,c/\sqrt{n}\bigr\}=
e^{\beta^Tz^{(i)}}\biggl(1-z_j^{(i)}F(u)\frac{c}{\sqrt{n}}+o_P\biggl(\frac{1}{\sqrt{n}}\biggr)\biggr).
\]

Let us consider the stochastic process
\[
Q_n(t,\beta)=\frac{1}n\bigl[\ell(t,\beta)-\ell(t,{\beta}_0)\bigr],
\]
where
\[
\ell(t,\beta)=
\sum_{i=1}^n\int_0^t\bigl\{\beta^Tz^{(i)}- \ln S^{(0)}(u,\beta)\bigr\}\,dN_i(u).
\]

Note that the derivative with respect to $\beta$ is
\[
\dot Q_n(t,\beta)
=\frac{1}n \dot\ell(t,\beta)=\frac{1}n
\sum_{i=1}^n\int_0^t\bigl\{z^{(i)}-E(u,\beta)\bigr\}dN_i(u).
\]
Constructing the test statistic we defined $\hat\beta$ as a random
vector satisfying the estimating equation
$\dot\ell(\tau,\beta)=0$. Let us show that $\hat\beta$ is a consistent
estimator of $\beta_0$.

\begin{thm}\label{thm2} If Assumptions \ref{assumpA} are satisfied then
the probability that the equation $\dot\ell(\beta)=0$ has a unique
solution, converges to $1$ and $\hat{\beta}
\stackrel{P}\rightarrow{\beta}_0$.
\end{thm}
\begin{proof}
Fix $\beta\in{\bf R}^m$.
Using the Doob--Meier decomposition of $N_i(t)$ write the stochastic
process $Q_n$ in the form
\begin{align*}
Q_n(t,\beta)&{}=\frac{1}n\sum_{i=1}^n\int_0^t\biggl[(\beta-{\beta}_0)^Tz^{(i)}-
\ln\frac{S^{(0)}(u,\beta)}{S^{(0)}(u,{\beta}_0)}\biggr]\, dN_i(u)
\\
&{}=\frac{1}n\sum_{i=1}^n\int_0^t\biggl[(\beta-{\beta}_0)^Tz^{(i)}-
\ln\frac{S^{(0)}(u,\beta)}{S^{(0)}(u,{\beta}_0)}\biggr]
\\
&\quad{}\times Y_i(u) e^{\beta_0^Tz^{(i)}}\biggl(1-\frac{c}{\sqrt
{n}}F(u)z_j^{(i)}+o_P\biggl(\frac{1}{\sqrt{n}}\biggr)\biggr)d\Lambda(u)
\\
&\quad{}+\frac{1}n\sum_{i=1}^n\int_0^t\biggl[(\beta-{\beta}_0)^Tz^{(i)}-
\ln\frac{S^{(0)}(u,\beta)}{S^{(0)}(u,{\beta}_0)}\biggr]\, dM_i(u)
\\
&{}
=\int_0^t\biggl[(\beta-{\beta}_0)^TS^{(1)}(u,{\beta}_0)- \ln
\frac{S^{(0)}(u,\beta)}{S^{(0)}(u,{\beta}_0)}S^{(0)}(u,{\beta}_0)\biggr]\,
d\Lambda(u)
\\
&\quad{}
+o_P(1)+ \frac{1}n\sum_{i=1}^n\int_0^t\biggl[(\beta-{\beta}_0)^Tz^{(i)}- \ln
\frac{S^{(0)}(u,\beta)}{S^{(0)}(u,{\beta}_0)}\biggr]\,
dM_i(u)\\
&{}=K_n(t,\beta)+\tilde M_n(t)+o_P(1).
\end{align*}
The predictable variation\index{predictable variation} of
the martingale $\tilde M_n$ is
\begin{align*}
\langle\tilde M_n
\rangle(t) &{}=\frac{1}{n^2}\sum_{i=1}^n\int_0^t\biggl\{(\beta-{\beta}_0)^Tz^{(i)}-
\ln
\frac{S^{(0)}(u,\beta)}{S^{(0)}(u,{\beta}_0)}\biggr\}^2\,Y_i(u)e^{{\beta
}_0^Tz^{(i)}}
\\
&\quad{}\times\biggl(1-\frac{c}{\sqrt{n}}F(u)z_j^{(i)}+o_P\biggl(\frac{1}{\sqrt
{n}}\biggr)\biggr)d\Lambda(u)\\
&{}=
\frac{1}{n^2}\int_0^t\biggl\{(\beta-{\beta}_0)^T
 S^{(2)}(u,{\beta}_0)(\beta-{\beta}_0)
\\
&\quad{}
-2(\beta-{\beta
}_0)^T
S^{(1)}(u,{\beta}_0)
\ln\frac{S^{(0)}(u,\beta)}{S^{(0)}(u,{\beta}_0)}
\\
&\quad{}
+
S^{(0)}(u,{\beta}_0)\ln^2
\frac{S^{(0)}(u,\beta)}{S^{(0)}(u,{\beta}_0)} \biggr\}d\Lambda(u)+o_p\biggl(\frac{1}n\biggr).
\end{align*}
So $\langle\tilde M_n
\rangle(\tau)\stackrel{P}\rightarrow0$.
$K_n(\tau,\beta)\stackrel{P}\rightarrow k(\beta)$;
here
\[
k(\beta)=\int_0^\tau\biggl[(\beta-{\beta}_0)^Ts^{(1)}(u,{\beta}_0)-
\ln
\frac{s^{(0)}(u,\beta)}{s^{(0)}(u,{\beta}_0)}s^{(0)}(u,{\beta
}_0)\biggr]d\Lambda(u),
\]
which implies
\[
Q_n(\beta)=Q_n(\tau,\beta)\stackrel{P}\rightarrow
k(\beta).
\]
Note that
\begin{gather*}
\dot k(\beta)=\int_0^\tau\bigl[s^{(1)}(u,{\beta}_0)-
e(u,\beta)s^{(0)}(u,{\beta}_0)\bigr]\,d\Lambda(u),\quad\dot
k({\beta}_0)=0.
\\
\ddot k(\beta)=-\int_0^\tau
v(u,\beta)s^{(0)}(u,{\beta}_0)\,d\Lambda(u),\quad\ddot k(\beta
_0)=-\Sigma(\beta_0).
\end{gather*}
So the matrix $\ddot k(\beta_0)$ is negatively definite. The remaining
part of the proof coincides with the proof of analogous theorem for the
Cox model\index{Cox model} (see Theorem VII.2.1 of Andersen \xch{et\,al.}{et al}), i.e.
Andersen's and Gill's theorem~\cite{AndGill1982} is applied. This
theorem says that if the
sequence of concave differentiable
stochastic processes $Q_n(\beta)$ pointwise converges
in probability to a real function $k(\beta)$ on a convex open set
$E\subset{\bf R}^m$, then: a) the function $k(\beta)$ is concave on
$E$; b) the convergence is uniform in probability on compact
subsets of the set $E$; c) if the function $k(\beta)$ has a unique
maximum at the
point ${\beta}_0$ then the probability that the equation $\dot
Q(\beta)=0$ has a unique root $\hat{\beta}$
in the set $E$ tends to $1$ and $\hat{\beta}\stackrel{P}\rightarrow
{\beta}_0$.

Assumption~\ref{assumpA}\,e) was crucial for application of Andersen's and Gill's theorem.
\end{proof}

Note that $\hat\Lambda$ is uniformly consistent estimator of $\Lambda$
on $[0,\tau]$:
\begin{align*}
\hat\Lambda(t)&{} =\int_0^t\frac{dN(u)}{S^{(0)}(u,\hat\beta)}=\Lambda(t)+
\int_0^t\biggl[\frac{S^{(0)}(u,\beta_0)}{S^{(0)}(u,\hat\beta)}-1\biggr]d\Lambda(u)
\\
&\quad{}-\frac{c}{\sqrt{n}}
\int_0^t\frac{S_j^{(1)}(u,\beta_0)}{S^{(0)}(u,\hat\beta)}d\Lambda
(u)+\int_0^t\frac{dM(u)}{S^{(0)}(u,\hat\beta)}+o_p(1)=\Lambda(t)+o_p(1)
\end{align*}
uniformly on $[0,\tau]$ because $\hat\beta$ is consistent and
\[
\biggl\langle\int_0^\tau\frac{dM(u)}{S^{(0)}(u,\hat\beta)}\biggr\rangle=\int_0^\tau
\frac{S^{(0)}(u,\beta_0)}{(S^{(0)}(u,\hat\beta))^2}d\Lambda(u)=o_p(1).
\]
Set
\[
\mu_j(t)=-
{c}\int_0^tv_j(u,\beta_0)s^{(0)}(u,\beta_0)F(u)d\Lambda(u).
\]

\begin{thm}\label{thm3} If Assumptions \ref{assumpA} are satisfied then
\begin{enumerate}
\item[1)]
\[
\frac{1}{\sqrt{n}}\dot\ell(\cdot,{\beta}_0)\stackrel{\cal
D}{\rightarrow} \mu_j(t)+Z(\cdot,{\beta}_0)\; \mbox{on}\;(D[0,\tau])^m,
\]
where $Z$
is an $m$-dimensional Gaussian process with components having independent
increments,
$Z_j(0)=0$ a.s. and for all $0 \leq s \leq t \leq\tau$:
\[
{\bf cov}(Z_j(s),Z_{j'}(t))=\sigma_{jj'}(s);
\]
here $\sigma_{jj'(t)}$ are the elements of the matrix $\Sigma(t)$.
In particular,
\[
\frac{1}{\sqrt{n}}\dot\ell({\beta}_0)\stackrel{\cal D}{\rightarrow}
\mu_j(\tau)+Z(\tau,{\beta}_0)\sim N(\mu_j(\tau),\Sigma(\tau)),\quad
\mbox{as}\ n\to\infty.
\]
\item[2)] $\hat\Sigma(\tau)= -\frac{1}n \ddot\ell(\hat{\beta})\stackrel
{P}\rightarrow\Sigma(\tau)$,
\item[3)] $\sqrt{n}(\hat{\beta}-{\beta}_0)=(-\frac{1}n\ddot\ell({\hat\beta
}))^{-1}\frac{1}{\sqrt{n}}\dot\ell({\beta}_0)+o_P(1)$.
\end{enumerate}
\end{thm}

\begin{proof}
Set
\[
M^*(t)=\frac{1}{\sqrt{n}}
\sum_{i=1}^n\int_0^t\bigl\{z^{(i)}-E(u,{\beta}_0)\bigr\}dM_i(u).
\]
Using the Doob--Meier decomposition we have
\begin{align*}
\frac{1}{\sqrt{n}}\dot\ell(t,{\beta}_0)&{}=
\frac{1}{\sqrt{n}}\sum_{i=1}^n\int_0^t\bigl\{z^{(i)}-E(u,{\beta}_0)\bigr\}dN_i(u)\\
&{}=M^*(t)+
\frac{1}{\sqrt{n}}\sum_{i=1}^n\int_0^t\bigl\{z^{(i)}
-
E(u,\beta_0)\bigr\} Y_i(u) e^{\beta_0^Tz^{(i)}}\\
&\quad{}\times\biggl(1-\frac{c}{\sqrt{n}}F(u)z_j^{(i)}+o_P\biggl(\frac{1}{\sqrt{n}}\biggr)\biggr)d\Lambda(u)
\\
&{}=
M^*(t)-
\frac{c}n\sum_{i=1}^n\int_0^tV_j(u,\beta_0)S^{(0)}(u,\beta
_0)F(u)d\Lambda(u)+o_P(1)
\\
&{}=
M^*(t)-
{c}\int_0^tv_j(u,\beta_0)s^{(0)}(u,\beta_0)F(u)d\Lambda
(u)+o_P(1)\\
&{}=M^*(t)+\mu_j(t)+o_P(1),
\end{align*}
where $V_j$ is the $j$th column of the matrix $V$.

The first term converges weakly to $Z(\cdot,{\beta}_0)\; \mbox{on}\;
(D[0,\tau])^m$ because the limit in probability of
its predictable covariation matrix has the same expression as that of
analogous term under the Cox model:\index{Cox model}
\begin{align*}
\langle M^*\rangle(t)&{}=
\frac{1}{{n}}\int_0^t
V(u,{\beta}_0)S^{(0)}(u,{\beta}_0)\,d\Lambda(u)-\frac{c}{n^{3/2}}
\sum^n_{i=1}
\int_0^t\bigl\{z^{(i)}-{E}(u,{\beta}_0)\bigr\}\\
&\quad{}\times\bigl\{z^{(i)}-{E}(u,{\beta}_0)\bigr\}^TY_i(u)\,e^{{\beta}_0^T
z^{(i)}}\bigl(F(u)z_j^{(i)}+o_P(1)\bigr)d\Lambda(u)
\\
&{}=\frac{1}{{n}}\int_0^t
V(u,{\beta}_0)S^{(0)}(u,{\beta}_0)\,d\Lambda(u)+o_P(1)\stackrel{
P}{\rightarrow} \Sigma(t).
\end{align*}
Analogously, verification of the Lindeberg's condition
is done by the same way as in the case of the Cox model.\index{Cox model}

Let us consider the norm of the difference:
\begin{align*}
\biggl\Vert-\frac{1}n \ddot\ell(\hat{\beta})-\Sigma(\tau)
\biggr\Vert
&{}\leq
\sup_{u\in[0,\tau]}\bigl|V(u,\hat{\beta})-v(u,\hat{\beta})\bigr|+\sup_{u\in[0,\tau]}\bigl|v(u,\hat{\beta})-v(u,{\beta}_0)\bigr|
\\
&\quad{}+
\biggl\Vert\frac{1}n\int_0^\tau v(u,{\beta}_0)dM(u)
\biggr\Vert\\
&\quad{}+
\biggl\Vert\int_0^\tau v(u,{\beta}_0)\biggl(\frac{1}n S^{(0)}(u,{\beta}_0)-s^{(0)}(u,{\beta}_0)\biggr)\,d\Lambda(u)\biggr\Vert
\\
&\quad{}
+
\biggl\Vert\frac{c}{n\sqrt{n}}\int_0^\tau F(u)\sum_{i=1}^nz_j^{(i)}e^{\beta_0^Tz^{(i)}}(1+o_p(1))d\Lambda(u)
\biggr\Vert.
\end{align*}
Using the fact that the estimator $\hat\beta$ is consistent and that
the first four terms have the same structure as analogous terms in the proof
of Theorem VII.2.2 of Andersen \xch{et\,al.}{et all}, we have that the first four
terms converge in probability to zero. Such convergence of the last
term is obvious
because the last term multiplied by $\sqrt{n}$ converges to a finite
limit in probability.

The mean value theorem and consistency of the estimator $\hat\beta$ imply
\[
\dot\ell_j(\hat{\beta})-\dot\ell_j({\beta}_0)=\ddot\ell_j({\beta
}_j^*)(\hat{\beta}-{\beta}_0)=n\biggl(\frac{1}n\ddot\ell_j({\hat\beta
})+o_P(1)\biggr)(\hat{\beta}-{\beta}_0);
\]
here ${\beta}_j^*$ is a point on the line segment joining the points
$\hat{\beta}$ and ${\beta}_{0}$. Since $\dot\ell_j(\hat{\beta})=0$, we
obtain
\[
\sqrt{n}(\hat{\beta}-{\beta}_0)=\biggl(-\frac{1}n\ddot\ell({\hat\beta
})+o_P(1)\biggr)^{-1}\frac{1}{\sqrt{n}}\dot\ell({\beta}_0)=
\biggl(-\frac{1}n\ddot\ell({\hat\beta})\biggr)^{-1}\frac{1}{\sqrt{n}}\dot\ell({\beta}_0)+o_p(1).\qedhere
\]
\end{proof}

Set
\begin{align*}
d&{}=c\biggl[ \int_0^{\tau}F^2(u)v_{jj}(u,\beta_0)s^{(0)}(u,\beta_0)d\Lambda(u)\\
&\quad{}-\int_0^{\tau} F(u\xch{)}{))}v^T_j(u,\beta_0)d\Lambda(u)\Sigma^{-1}(\tau)\int_0^\tau F(u)v_j(u,\beta_0)s^{(0)}(u,\beta_0)d\Lambda(u)\biggr].
\end{align*}
Note that for $m=1$
\[
d/c=\frac{\int F^2vs^{(0)}d\Lambda\int vs^{(0)}d\Lambda-(\int
Fvs^{(0)}\xch{d}{)d}\Lambda)^2}{
\int_vs^{(0)}d\Lambda}>0,
\]
because $F$ is not equal to 1 a.s. on $[0,\tau]$.
\begin{thm}\label{thm4} Under Assumptions \ref{assumpA}
\[
T\stackrel{d}\rightarrow\mu+N(0,1)=N(\mu,1),
\]
where $\mu=d/ D_j$ and $D_j$ is
the limit in probability of the random variable $\hat D_j$.
\end{thm}
\begin{proof} Set
\[
\bar M(t)=n^{-1/2}\sum_{i=1}^n\int_0^{t}\hat F(u)
\bigl\{z_j^{(i)}-E_{j}(u,\hat\beta)\bigr\} dM_i(u).
\]
Similarly as in the proof of Theorem~\ref{thm1} using the equality
$S_j^{(1)}-E_{j}S^{(0)}=0$, consistency of $\hat\beta$, Assumptions~\ref{assumpA},
uniform consistency of $\hat F$ on $[0,\tau]$, we write the following
expression:
\begin{align*}
n^{-1/2}\hat U_{c/\sqrt{n}}(t)
&{}=-n^{-1/2}\sum_{i=1}^n\int_0^{t}\hat F(u)
\bigl\{z_j^{(i)}-E_{j}(u,\hat\beta)\bigr\} dN_i(u)\\
&{}=-\bar M(t)
-n^{-1/2}
\sum_{i=1}^n\int_0^{t}\hat F(u)\bigl\{z_j^{(i)} -E_{j}(u,\hat\beta)\bigr\}e^{\beta
_0^Tz^{(i)}}\\
&\quad{}\times
\bigl\{
1-z_j^{(i)}
F(u)\frac{c}{\sqrt{n}}(1+o_p(1))\bigr\}
Y_i(u)d\Lambda(u)
\\
&{}
=
-\bar M(t)+ \frac{c}n \int_0^{t}\hat F(u)F(u)V_{jj}(u,\beta
_0)S^{(0)}(u,\beta_0)d\Lambda(u) \\
&\quad{}
+
n^{-1} \int_0^{t}\hat F(u)V_j^T(u,\beta_0)S^{(0)}(u,\beta_0) d\Lambda
(u)\,n^{1/2}(\hat\beta-\beta_0)+o_P(1)
\end{align*}
uniformly on $[0,\tau]$. Applying Theorem~\ref{thm3} write the right side in the form
\begin{align*}
n^{-1/2}\hat U_{c/\sqrt{n}}(t)&{}=c\biggl[ \int_0^{t}F^2(u)v_{jj}(u,\beta
_0)s^{(0)}(u,\beta_0)d\Lambda(u)
\\
&\quad{}
-\int_0^{t} F(u\xch{)}{))}v^T_j(u,\beta_0)d\Lambda(u)\Sigma^{-1}(\tau)\\
&\quad{}\times\int_0^\tau F(u)v_j(u,\beta_0)s^{(0)}(u,\beta_0)d\Lambda(u)\biggr]
\\
&\quad{}
- n^{-1/2}\sum_{i=1}^n\int_0^{t} \hat F(u) \bigl\{z_j^{(i)}-E_{j}(u,\hat\beta
)\bigr\}dM_i(u)
\\
&\quad{}
- \hat\Sigma^T_{j}(t)\hat\Sigma^{-1}(\tau)\, n^{-1/2}\sum^n_{i=1}\int^{\tau}_0\bigl\{z^{(i)}-{E}(u,\hat\beta)\bigr\}dM_i(u) +o_p(1)
\end{align*}
uniformly on $[0,\tau]$.

Note that the non-martingale part is $d$ and the martingale part of the
expression is exactly
of the same form as in the case of the PH model
and has the same limit distribution as in latter case.
\end{proof}

The power function of the test against approaching alternatives is
\[
\beta=2-\Phi(z_{\frac{\alpha}2}-\mu)-\Phi(z_{\frac{\alpha}2}+\mu),
\]
where $\Phi$ and $z_{\frac{\alpha}2}$ are the c.d.f. and the upper
$\alpha/2$ critical value of the standard normal law, respectively.
If $d\neq0$ and $c$ is large then $\mu$ is large and the power is near
to 1.

\section{Case of several covariates}

The proportional hazards hypothesis for several covariates
$z_{j_1},\ldots,z_{j_k}$ (for all covariates $z_1,\ldots,z_m$, in particular)
is tested similarly.

Set $\bar z=(z_{j_1},\ldots,z_{j_k})^T$, $\gamma=(\gamma_{j_1},\ldots,\gamma
_{j_k})^T$.
The term $\gamma_j z_j$
is replaced by $\gamma^T \bar z$ in the model (\ref{model}).

Replacing $z^{(i)}_j$ by $\bar z^{(i)}$,
$E_{j}(t,\hat\beta)$ by $(E_{j_1}(t,\hat\beta),\ldots,E_{j_k}(t,\hat\beta
))^T$ in the statistic $\hat U_{\gamma_j}$ we obtain a statistic
denoted by $\hat U_\gamma$.
The statistic $T$ has the form (\ref{statistics}), where
the element $V_{jj}(u,\hat\beta)$ is replaced by the $k\times k$ matrix
$(V_{j_l,j_s}(u,\hat\beta))_{k\times k}$, the vector $V_{j}(u,\hat\beta
)$ is replaced by
the $m\times k$ matrix $(V_{i,j_s}(u,\hat\beta))$, $i=1,\ldots,m$,
$s=1,\ldots,k$ in the definitions of $\Sigma_{jj}$, $\Sigma_{j}$, and
$\hat U_{\gamma_j}$ is replaced by $\hat U_\gamma$.

So
$\hat\Sigma_{jj}(t)$, $\hat\Sigma_{j}(t)$, and $\hat D$ become $k\times
k$, $m\times k$, and $k\times k$ matrices, respectively,
and $U_\gamma$ becomes a $k\times1$ vector.

The test statistic
\[
T=n^{-1}\hat U^T_\gamma\hat
D^{-1}\hat U_\gamma\stackrel{\cal{D}}\rightarrow\chi^2(k)\quad\mbox{
as}\ n\to\infty.
\]
The hypothesis $H_0: \gamma=0$ is rejected with approximate
significance level $\alpha$ if $T> \chi^2_\alpha(k)$;
here $\chi^2_\alpha(k)$ is the  $\alpha$ critical value of the standard
chi-squared distribution with $k$ degrees of freedom.

\section{Real data analysis}

\begin{exmp}[Chemo-radio data, one-dimensional dichotomous
covariate]
Stablein and Koutrouvelis~\cite{Stab1985} studied the
well-known two-sample data of the Gastrointestinal Tumor Study Group concerning
effects of chemotherapy ($z=0$) and chemotherapy plus radiotherapy
($z=1$) on the survival times of gastric cancer patients.

The number of patients $n=90$. The data are right-censored.
The value of the test statistic $T$ is 3.651, the $p$-value is 0.0003.
The Cox model\index{Cox model} is rejected.
The result is natural because the Kaplan--Meier estimators of the
survival functions of two patient groups intersect.
\end{exmp}

\begin{exmp}[Prisoners data, 7 covariates]
The data are given in~\cite{Allison2010}.
They consist of 432 male inmates who
were released from Maryland state prisons in the early 1970s. These men
were followed for 1 year after their release, and the dates
of any arrests were recorded. Time is measured by the week of the first
arrest after release.
There were seven covariates: financial aid after release (FIN; 0~-- no,
1~-- yes), age in years at the time of release (AGE),
race\index{race} (RACE;\index{race} 1~-- black, 0~-- otherwise), full-time work experience before
incarceration\index{incarceration} (WEX; 1~-- yes, 0~-- no),
marital status (MAR; 1~-- was married at the time of release, 0~-- otherwise),
released on parole\index{parole} (PAR; 1~-- released on parole,\index{parole} 0~-- otherwise),
convictions prior to incarceration\index{incarceration} (PRI).

The results of testing hypothesis for all covariates: the value of the
test statistic $T$ is 17.58, the $p$-value is 0.014.
The assumption of the PH assumption is rejected.

The results for each covariate\index{covariate} are given in Table~\ref{tbl1}. The
PH assumption is rejected for AGE and WEXP covariates.\index{WEXP covariates}
\begin{table}[h]
\caption{Prisoners data. The values of test statistics and $p$-values}\label{tbl1}
\begin{tabular}{@{}lrrrrrrrr@{}}
\hline
Covar.   & FIN   & AGE   & RACE  & WEX      & MAR      & PAR      & PRI   & Glob. test \\
Stat.    & 0.162 & 2.464 & 1.423 & $-$2.033 & $-$1.017 & $-$0.222 & 0.672 & 17.58      \\
$p$-val. & 0.872 & 0.014 & 0.155 & 0.042    & 0.309    & 0.824    & 0.502 & 0.014      \\
\hline
\end{tabular}
\end{table}
\end{exmp}

\begin{exmp}[UIS dataset, 10 covariates]
Let us consider right censored UIS data set
given in~\cite{HosLemMay2008}.

UIS was a 5-year research project comprised of two concurrent
randomized trials
of residential treatment for drug abuse.\index{residential treatment for drug abuse} The purpose of the study was
to compare treatment programs of different planned durations designed
to reduce drug abuse\index{drug abuse} and to prevent high-risk HIV behavior.
The UIS sought to determine whether alternative residential treatment\index{residential treatment} approaches
are variable in effectiveness and whether efficacy depends on planned
program duration.
The time variable is time to return to drug use\index{drug use} (measured from
admission). The individuals who did not returned to drug use\index{drug use} are right censored.
We use the model with 10 covariates (which support PH assumption) given
by Hosmer, Lemeshow and May (2008).
The covariates\index{covariate} are:
AGE (years);
Beck depression score (beckt; 0~-- 54);
$\mathrm{NDR}1=((\mathrm{NDR}+1)/10)^{(-1)}$;
$\mathrm{NDR}2=((\mathrm{NDR}+1)/10)^{(-1)}\log((\mathrm{NDR}+1)/10)$;
drug use\index{drug use} history at admission
(IVHX\_3; 1~-- recent, 0~-- never or previous);
RACE\index{race} (0~-- white, 1~-- non-white);
treatment randomization assigment
(TREAT; 0~-- short, 1~-- long);
treatment site (SITE; 0~-- A, 1~-- B);
interaction of age and treatment site (AGEXS);
interaction of race\index{race} and treatment site (RACEXS\index{RACEXS}).
The NDR denotes number of prior drug treatments (0~-- 40).
Due to missing data in covariates,\index{covariate} the model is based on 575 of the 628
observations.

The results of testing hypothesis for all covariates: the value of the
test statistic $T$ is 6.781, the $p$-value is 0.7460.
The assumption of the PH assumption is not rejected.

The results for each covariate\index{covariate} are given in Table~\ref{tbl2}.

\begin{table}[h]
\caption{UIS dataset. The values of test statistics and $p$-values}\label{tbl2}
\tabcolsep=0pt
\begin{tabular*}{0.75\textwidth}{@{\extracolsep{\fill}}lrrrrrr@{}}
\hline
Covariate & AGE      & beckt & NDR1  & NDR2  & IVHX\_3 & RACE     \\
Stat.     & $-$0.061 & 1.085 & 0.182 & 0.118 & 0.912   & $-$1.278 \\
$p$-value & 0.952    & 0.278 & 0.856 & 0.906 & 0.362   & 0.201    \\
\hline
\end{tabular*}\\[9pt]
\begin{tabular*}{0.75\textwidth}{@{\extracolsep{\fill}}lrrrrr@{}}
\hline
Covariate & TREAT & SITE  & AGEXS    & RACEXS & Global test \\
Stat.     & 0.792 & 1.016 & $-$0.378 & 6.781  & $-$0.107    \\
$p$-value & 0.429 & 0.309 & 0.705    & 0.746  & 0.915       \\
\hline
\end{tabular*}
\end{table}

The assumption of the PH assumption for individual covariates is not rejected.
\end{exmp}

\begin{acknowledgement}[title={Acknowledgement}]
We thank the anonymous reviewers for their careful reading of our
manuscript and their many
insightful comments and suggestions.
\end{acknowledgement}

\end{document}